\documentclass[runningheads]{llncs}

\usepackage[utf8]{inputenc}
\usepackage[T1]{fontenc} 

\usepackage{lmodern}

\usepackage{amsmath}        
\usepackage{amsfonts}       
\usepackage{amssymb}
\usepackage{cleveref}
\usepackage{subcaption}     
\usepackage{caption}

\usepackage{tikz}			
\usetikzlibrary{decorations.pathreplacing,calc,arrows,arrows.meta,patterns,ipe}

\newif{\ifarxiv}
\arxivtrue

\newcommand{\arxivorgd}[2]{\ifarxiv #1\else #2\fi}

\newcommand{\precg}[1]{\ensuremath{({=}\,#1)}}
\newcommand{\bbN}{\mathbb{N}}

\begin{document}
\title{String graphs with precise number of intersections\thanks{The first author was supported by the Czech Science Foundation Grant No. 19-27871X. The second author was supported by the Czech Science Foundation Grant No. 23-04949X.}}


\author{Petr Chmel\inst{1}\orcidID{0000-0002-9131-1458} \and
Vít Jelínek\inst{1}\orcidID{0000-0003-4831-4079}}

\authorrunning{P. Chmel and V. Jelínek} 

\institute{Computer Science Institute, Charles University, Prague, Czech Republic\\
\email{\{chmel,jelinek\}@iuuk.mff.cuni.cz}}

\maketitle

\begin{abstract}
A string graph is an intersection graph of curves in the plane.
A $k$-string graph is a graph with a string representation in which every pair of curves intersects in at most $k$ points.
We introduce the class of $\precg{k}$-string graphs as a further restriction of $k$-string graphs by requiring that every two curves intersect in either zero or precisely $k$ points.
We study the hierarchy of these graphs, showing that for any $k\geq 1$, $\precg{k}$-string graphs are a subclass of $\precg{k+2}$-string graphs as well as of $\precg{4k}$-string graphs; however, there are no other inclusions between the classes of \precg{k}-string and \precg{\ell}-string graphs apart from those that are implied by the above rules.
In particular, the classes of $\precg{k}$-string graphs and $\precg{k+1}$-string graphs are incomparable by inclusion for any $k$, and the class of $\precg{2}$-string graphs is not contained in the class of $\precg{2\ell+1}$-string graphs for any~$\ell$.
\end{abstract}


\section{Introduction}
An \emph{intersection representation} of a graph $G=(V,E)$ is a collection of sets 
$R=\{R(v)\colon v\in V\}$ such that for every two distinct vertices $u,v\in V,$ the sets
$R(u),R(v)$ intersect if and only if $u,v$ form an edge of~$G$. In this paper, we 
focus on \emph{string representations}, which are intersection representations where the sets 
$R(v)$ are curves in the plane. Graphs admitting a string representation are known as \emph{string 
graphs}, and their class is denoted by \textsc{string}. They were introduced by 
Sinden~\cite{SindenStringGraphIntroduction}.

A common way to further restrict string representations is to bound the number of intersections 
between a pair of curves -- this gives rise to the class of 
\emph{$k$-string graphs} (denoted by \textsc{$k$-string}), which are the graphs admitting a string
representation in which every two curves intersect in at most $k$ points.
This yields a hierarchy of classes parameterized by the number $k$. Clearly, \textsc{$k$-string} 
is a subclass of \textsc{$(k+1)$-string} for all $k\geq 1$, and Kratochvíl and 
Matoušek~\cite{KratochvilMatousekSEG} showed that the inclusions are strict.
Moreover, they showed that it is NP-complete to recognize $k$-string graphs for 
any fixed $k\geq 1$~\cite{KratochvilMatousekNPhard}.

A folklore result shows that any string graph can be represented as an intersection graph of 
paths on a rectilinear grid; this is known as a VPG representation. 
Asinowski et~al.~\cite{AsinowskiEtAlVPG} introduced the classes of B$_k$-VPG graphs, which are 
graphs admitting a VPG representation in which every path has at most $k$ bends. Clearly, \textsc{B$_k$-VPG}
is a subclass of \textsc{B$_{k+1}$-VPG}.
Chaplick et~al.~\cite{ChaplickEtAlNoodles} have shown that these inclusions are strict, and that 
for any $k\geq 1$, it is NP-complete to recognize the class \textsc{B$_k$-VPG}; for $k=0$, 
NP-completeness was proven by Kratochvíl~\cite{KratochvilGIGNPhard,KratochvilMatousekNPhard}.

By way of context, we also mention other hierarchies, such as $k$-DIR 
graphs~\cite{KratochvilMatousekSEG}: the intersection graphs of line segments in the plane with at 
most $k$ different slopes, $k$-length-segment graphs~\cite{CabelloJejcicHierarchies}: the 
intersection graphs of line segments in the plane with at most $k$ different lengths, $k$-size-disk 
graphs~\cite{CabelloJejcicHierarchies}: the intersection graphs of disks in the plane with at most 
$k$ different diameters, and $k$-interval graphs~\cite{TrotterHararyKIntervalGraphs}: the 
intersection graphs of unions of $k$ intervals on the real line.

We can also find examples of such hierarchies within contact graphs, where intersections are only permitted when the intersection point is an endpoint of at least one of the intersecting curves.
We may restrict the number of curves sharing a single point, which yields the class of $k$-contact graphs~\cite{HlinenyCGHierarchy} -- this class can be also restricted further to only line segments and not general strings.
Additionally, we can define the class \textsc{B$_k$-CPG} ~\cite{DenizEtAlCPGHierarchy}, which is analogous to \textsc{B$_k$-VPG}, except the graphs must be contact graphs.
Recognition of graphs from these classes is also known to be 
NP-complete~\cite{HlinenyCGRecognitionNPComplete,ChampseixEtAlBkCPGRecognitionNPComplete}.

In all of the above cases, the graph classes in the hierarchy form a chain of strict inclusions:
increasing the parameter $k$ by one yields a strict superclass.

We introduce a new hierarchy of so called $\precg{k}$-string graphs, which are the intersection graphs of 
curves in the plane such that every intersection is a crossing and every pair of intersecting curves
has \emph{exactly} $k$ crossings. This hierarchy turns out to have a more intricate structure than 
all the previously mentioned examples. For instance, while \textsc{\precg{k}-string} is a strict 
subclass of \textsc{\precg{k+2}-string} for any $k\ge 1$, the two classes \textsc{\precg{k}-string} 
and \textsc{\precg{k+1}-string} are incomparable by inclusion. With our results, we will
fully characterise all the inclusions between the classes of \precg{k}-string graphs, for $k\in\bbN$.

Our results show that the parity of $k$ plays a key role in understanding the relations between the 
classes of \precg{k}-string graphs. This motivates us to introduce \emph{odd-string} graphs, as 
graphs having a representation where any two intersecting curves have an odd number of intersections. 
We show that this class is a proper subclass of string graphs; indeed, there are \precg{2}-string 
graphs which are not odd-string graphs. 

We remark that this distinction between curves with odd and even number of intersections is somewhat 
reminiscent of the distinction between pair-crossing and odd-crossing numbers of graphs, which is 
relevant in the study of topological graph 
drawing~\cite{PachTothCrossingNumbers,PelsmajerSchaeferStefankovicCrossingNumbers,SchaeferCrossingNumberSurvey}.

Let us formally introduce our terminology.

\begin{definition}[Proper string representation]
A string representation $R$ is \emph{proper} if every curve in $R$ is a simple piecewise linear 
curve with finitely many bends, any two curves in $R$ intersect in at most finitely many points, no 
three curves in $R$ intersect in one point, and  no intersection point of two curves coincides with 
a bend or an endpoint of a curve. Note that this implies that every intersection of two curves is a 
crossing.
\end{definition}

\begin{definition}[$\precg{k}$-string graphs]
    A \emph{$\precg{k}$-string representation} is a proper representation in which any two curves 
are either disjoint or intersect in precisely $k$~points. A graph is a \emph{$\precg{k}$-string 
graph} if it has a $\precg{k}$-string representation. The class of such graphs is denoted by 
\textsc{\precg{k}-string}.
\end{definition}

We stress that a $\precg{k}$-string representation must, by definition, be proper. This prevents 
us, e.g., from increasing the number of intersection points of two curves by introducing contact 
points in which the curves touch but do not cross.

\section{The hierarchy}

Our main goal in this paper is to understand the inclusion hierarchy of the classes of 
$\precg{k}$-string graphs for various~$k\in\bbN$.  We begin by two simple results.

\begin{proposition}\label{prop:preciselystringplustwo}
    For all $k\ge 1$, \textsc{$\precg{k}$-string} is a subclass of \textsc{$\precg{k+2}$-string}.
\end{proposition}

\begin{proposition}\label{prop:preciselystringtimesfour}
    For all $k\ge 1$, \textsc{$\precg{k}$-string} is a subclass of \textsc{$\precg{4k}$-string}.
\end{proposition}

The proofs of these results follow from the operations depicted in 
Figure~\ref{fig:preciselystringchanges}, with Proposition~\ref{prop:preciselystringplustwo} using Figure~\ref{fig:preciselystringchangeplustwo} and Proposition~\ref{prop:preciselystringtimesfour} using Figure~\ref{fig:preciselystringchangetimesfour}. The full proofs appear \arxivorgd{in the appendix}{ in the full version~\ref{ChmelJelinekArXivVersion}}.

\begin{figure}
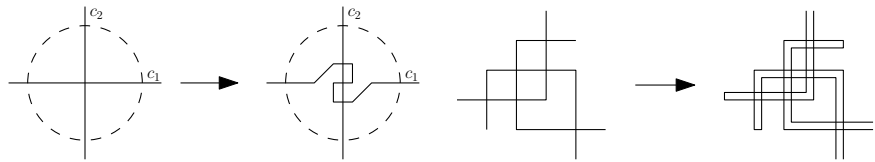

    \centering
    \begin{subfigure}[b]{0.45\linewidth}
        \centering
        \input{img/preciselystringplustwo.tex}
        \caption{Adding two intersections}
        \label{fig:preciselystringchangeplustwo}
    \end{subfigure}
    \hfil
    \begin{subfigure}[b]{0.45\linewidth}
        \centering
        \input{img/preciselystringtimesfour.tex}
        \caption{Quadrupling the intersections}
        \label{fig:preciselystringchangetimesfour}
    \end{subfigure}
    \caption{Two operations for increasing the number of intersection points}
    \label{fig:preciselystringchanges}
\end{figure}

We also note, for future reference, that for bipartite graphs, we may only double the representation 
of just one of the two partition classes.

\begin{proposition}\label{prop:preciselystringtimestwoforbipartite}
For all $k\ge 1$, every bipartite \precg{k}-string graph is also a \precg{2k}-string graph.
\end{proposition}

Note that Proposition~\ref{prop:preciselystringtimesfour} implies that any \precg{k}-string 
representation with $k$ odd can be transformed into a \precg{\ell}-string representation with $\ell$ 
even and large enough. However, there is no obvious way to proceed in the opposite direction, towards a 
representation with an odd number of crossings per pair of curves.
This motivates our next definition.
\begin{definition}[Odd-string graphs]
    An \emph{odd-string representation} of a graph $G$ is a proper string representation $R$ in 
which any two curves are either disjoint or cross in an odd number of points. A graph $G$ is an 
\emph{odd-string graph} if it has an odd-string representation. The class of odd-string graphs is 
denoted by \textsc{odd-string}.
\end{definition}
Of course, for any $k$ odd, \textsc{$\precg{k}$-string} is a subclass of \textsc{odd-string}. 
Conversely, given any odd-string representation $R$, we can repeatedly apply the operation from 
Figure~\ref{fig:preciselystringchangeplustwo} to transform $R$ into a \precg{k}-string 
representation for some $k$~odd. This shows that \textsc{odd-string} can equivalently be defined as 
the union $\bigcup_{\ell\in\bbN} \textsc{\precg{2\ell-1}-string}$.

\subsection{Hierarchy non-inclusions}

Our main results in this paper are negative results, which essentially show that there are no more 
inclusions between the classes we consider beyond those implied by 
Propositions~\ref{prop:preciselystringplustwo} and~\ref{prop:preciselystringtimesfour}. 
Specifically, we will prove the following results:
\begin{enumerate}
    \item for all $1\leq k<\ell$, the class of $\precg{\ell}$-string graphs is not a subclass of 
    $\precg{k}$-string graphs (see Theorem~\ref{thm:preciselykplusonenotink}),
    \item for all $k\geq 1$, the classes of $\precg{k}$-string graphs is not a subclass of $\precg{k+1}$-string graphs (see Theorem~\ref{thm:preciselyknotinkplusone}),
   \item for any odd $k$, the class of $\precg{k}$-string graphs is not a subclass of 
$\precg{4k-2}$-string graphs -- in particular, the bound $4k$ in 
Proposition~\ref{prop:preciselystringtimesfour} is best possible when it comes to transforming a 
\precg{k}-string representation with $k$ odd into a \precg{\ell}-string representation with $\ell$ 
even (see Theorem~\ref{thm-noteven4k}),
    \item there exist $\precg{2}$-string graphs with no odd-string representation; consequently,
\textsc{odd-string} is a proper subclass of \textsc{string} (see 
Theorem~\ref{thm:preciselynonoddstring}).
\end{enumerate}

Figure~\ref{fig:preciselystringinclusions} outlines the inclusion between the various 
\precg{k}-string classes. For reference, the figure also includes the classes \textsc{$k$-string}. 
We do not have a full characterisation of the inclusions between \textsc{$k$-string} and 
\textsc{\precg{\ell}-string}, although we will later show that \textsc{$k$-string} is a subclass 
of \textsc{\precg{8k}-string} (Theorem~\ref{thm:stringtopreciselykstring}).

\begin{figure}
    \centering
    \tikzstyle{ipe stylesheet} = [
  ipe import,
  even odd rule,
  line join=round,
  line cap=butt,
  ipe pen normal/.style={line width=0.4},
  ipe pen heavier/.style={line width=0.8},
  ipe pen fat/.style={line width=1.2},
  ipe pen ultrafat/.style={line width=2},
  ipe pen normal,
  ipe mark normal/.style={ipe mark scale=3},
  ipe mark large/.style={ipe mark scale=5},
  ipe mark small/.style={ipe mark scale=2},
  ipe mark tiny/.style={ipe mark scale=1.1},
  ipe mark normal,
  /pgf/arrow keys/.cd,
  ipe arrow normal/.style={scale=7},
  ipe arrow large/.style={scale=10},
  ipe arrow small/.style={scale=5},
  ipe arrow tiny/.style={scale=3},
  ipe arrow normal,
  /tikz/.cd,
  ipe arrows, 
  <->/.tip = ipe normal,
  ipe dash normal/.style={dash pattern=},
  ipe dash dotted/.style={dash pattern=on 1bp off 3bp},
  ipe dash dashed/.style={dash pattern=on 4bp off 4bp},
  ipe dash dash dotted/.style={dash pattern=on 4bp off 2bp on 1bp off 2bp},
  ipe dash dash dot dotted/.style={dash pattern=on 4bp off 2bp on 1bp off 2bp on 1bp off 2bp},
  ipe dash normal,
  ipe node/.append style={font=\normalsize},
  ipe stretch normal/.style={ipe node stretch=1},
  ipe stretch normal,
  ipe opacity 10/.style={opacity=0.1},
  ipe opacity 30/.style={opacity=0.3},
  ipe opacity 50/.style={opacity=0.5},
  ipe opacity 75/.style={opacity=0.75},
  ipe opacity opaque/.style={opacity=1},
  ipe opacity opaque,
]
\definecolor{red}{rgb}{1,0,0}
\definecolor{blue}{rgb}{0,0,1}
\definecolor{green}{rgb}{0,1,0}
\definecolor{yellow}{rgb}{1,1,0}
\definecolor{orange}{rgb}{1,0.647,0}
\definecolor{gold}{rgb}{1,0.843,0}
\definecolor{purple}{rgb}{0.627,0.125,0.941}
\definecolor{gray}{rgb}{0.745,0.745,0.745}
\definecolor{brown}{rgb}{0.647,0.165,0.165}
\definecolor{navy}{rgb}{0,0,0.502}
\definecolor{pink}{rgb}{1,0.753,0.796}
\definecolor{seagreen}{rgb}{0.18,0.545,0.341}
\definecolor{turquoise}{rgb}{0.251,0.878,0.816}
\definecolor{violet}{rgb}{0.933,0.51,0.933}
\definecolor{darkblue}{rgb}{0,0,0.545}
\definecolor{darkcyan}{rgb}{0,0.545,0.545}
\definecolor{darkgray}{rgb}{0.663,0.663,0.663}
\definecolor{darkgreen}{rgb}{0,0.392,0}
\definecolor{darkmagenta}{rgb}{0.545,0,0.545}
\definecolor{darkorange}{rgb}{1,0.549,0}
\definecolor{darkred}{rgb}{0.545,0,0}
\definecolor{lightblue}{rgb}{0.678,0.847,0.902}
\definecolor{lightcyan}{rgb}{0.878,1,1}
\definecolor{lightgray}{rgb}{0.827,0.827,0.827}
\definecolor{lightgreen}{rgb}{0.565,0.933,0.565}
\definecolor{lightyellow}{rgb}{1,1,0.878}
\definecolor{black}{rgb}{0,0,0}
\definecolor{white}{rgb}{1,1,1}
\begin{tikzpicture}[ipe stylesheet,scale=0.7, every node/.style={scale=0.7}]
  \draw
    (272, 816) rectangle (336, 784);
  \node[ipe node, anchor=center, font=\Large]
     at (304, 800) {\textsc{string}};
  \node[ipe node, anchor=center, font=\Large]
     at (304, 432) {(=1)\textsc{-string}};
  \node[ipe node, anchor=center, font=\Large]
     at (304, 528) {(=3)\textsc{-string}};
  \draw
    (304, 448)
     -- (304, 512);
  \draw
    (128, 608) rectangle (192, 576);
  \node[ipe node, anchor=center, font=\Large]
     at (160, 592) {2\textsc{-string}};
  \draw
    (128, 672) rectangle (192, 640);
  \node[ipe node, anchor=center, font=\Large]
     at (160, 656) {3\textsc{-string}};
  \draw
    (160, 608)
     -- (160, 640);
  \draw
    (128, 544) rectangle (192, 512);
  \node[ipe node, anchor=center, font=\Large]
     at (160, 528) {1\textsc{-string}};
  \draw
    (160, 544)
     -- (160, 576);
  \draw
    (128, 736) rectangle (192, 704);
  \node[ipe node, anchor=center, font=\Large]
     at (160, 720) {4\textsc{-string}};
  \draw
    (160, 672)
     -- (160, 704);
  \draw
    (160, 512)
     -- (304, 448);
  \draw
    (304, 544)
     -- (160, 640);
  \draw[ipe dash dashed]
    (160, 736)
     -- (160, 768)
     -- (304, 784);
  \draw
    (256, 544) rectangle (352, 512);
  \draw
    (256, 448) rectangle (352, 416);
  \node[ipe node, anchor=center, font=\Large]
     at (304, 752) {\textsc{odd-string}};
  \draw
    (256, 768) rectangle (352, 736);
  \node[ipe node, anchor=center, font=\Large]
     at (464, 720) {(=12)\textsc{-string}};
  \draw
    (512, 704) rectangle (416, 736);
  \node[ipe node, anchor=center, font=\Large]
     at (464, 624) {(=8)\textsc{-string}};
  \draw
    (416, 640) rectangle (512, 608);
  \node[ipe node, anchor=center, font=\Large]
     at (464, 528) {(=4)\textsc{-string}};
  \draw
    (512, 512) rectangle (416, 544);
  \node[ipe node, anchor=center, font=\Large]
     at (464, 432) {(=2)\textsc{-string}};
  \draw
    (416, 448) rectangle (512, 416);
  \draw
    (464, 448)
     -- (256, 496)
     -- (160, 576);
  \draw
    (464, 544)
     -- (160, 704);
  \draw
    (464, 704)
     -- (304, 544);
  \draw
    (464, 608)
     -- (160, 544);
  \draw[ipe dash dashed]
    (464, 640)
     -- (464, 704);
  \draw[ipe dash dashed]
    (464, 544)
     -- (464, 608);
  \draw
    (304, 448)
     -- (464, 512);
  \draw
    (464, 448)
     -- (464, 512);
  \draw[ipe dash dashed]
    (464, 736)
     -- (464, 768)
     -- (304, 784);
  \draw[ipe dash dashed]
    (304, 544)
     -- (304, 736);
  \draw[ipe dash dashed]
    (304, 768)
     -- (304, 784);
\end{tikzpicture}
    \caption{The diagram of inclusions of the investigated classes. Note that $k$-string 
representations are not assumed to be proper.}
    \label{fig:preciselystringinclusions}
\end{figure}
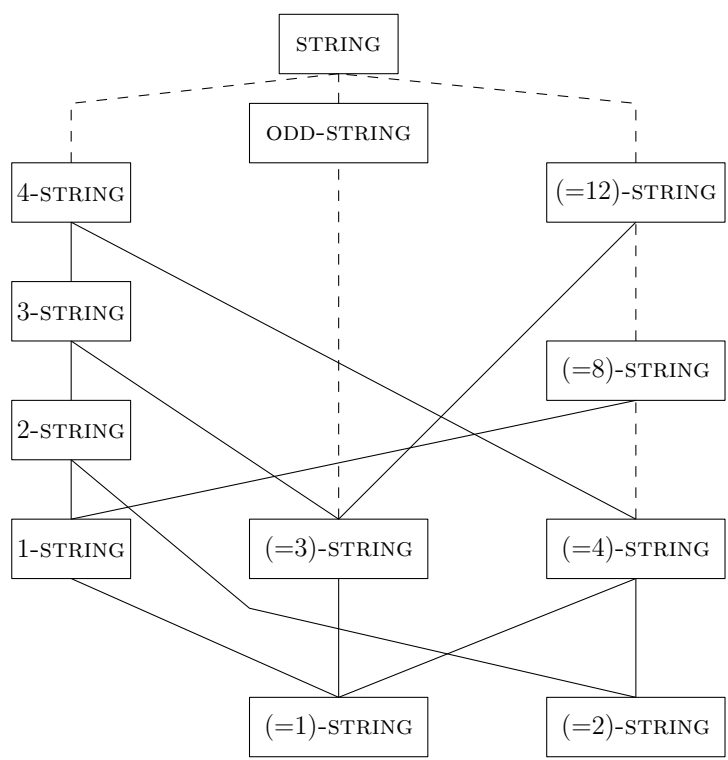

\subsubsection{Noodle-Forcing Lemma}

Our main tool to prove the non-inclusions in our hierarchy is the Noodle-Forcing Lemma of Chaplick et 
al. \cite{ChaplickEtAlNoodles}, originally devised to prove non-inclusions in the B$_k$-VPG 
hierarchy. We will show that the lemma can be adapted to the setting of $\precg{k}$-string 
representations. We begin by stating a simplified version of the lemma.

\begin{lemma}[Noodle-Forcing Lemma, Chaplick et al. \cite{ChaplickEtAlNoodles}]\label{lem:noodle}
    Let $G=(V,E)$ be a graph with a proper string representation $R=\{R(v)\colon v\in V\}$.
    Then there is a graph $G^\#(R)=(V^\#,E^\#)$ containing $G$ as an induced subgraph that has a 
proper string
representation $R^\#=\{R^\#(v)\colon v\in V^\#\}$ such that $R(v)=R^\#(v)$ for all $v\in V$ and $R^\#(w)$ is 
a vertical or a horizontal segment for $w\in V^\#\setminus V$.

    Moreover, for any $\varepsilon>0$ any (not necessarily proper) string representation of $G^\#$ 
can be 
transformed by a homeomorphism of the plane and a circular inversion into a representation 
$R^\varepsilon=\{R^\varepsilon(v)\colon v\in V^\#\}$ such that for every vertex $v\in V$,
the curve $R^\varepsilon(v)$ is contained in the $\varepsilon-$neighborhood of $R(v)$ and
$R(v)$ is contained in the $\varepsilon-$neighborhood of $R^\varepsilon(v)$.
\end{lemma}

\begin{figure}
    \centering
    \input{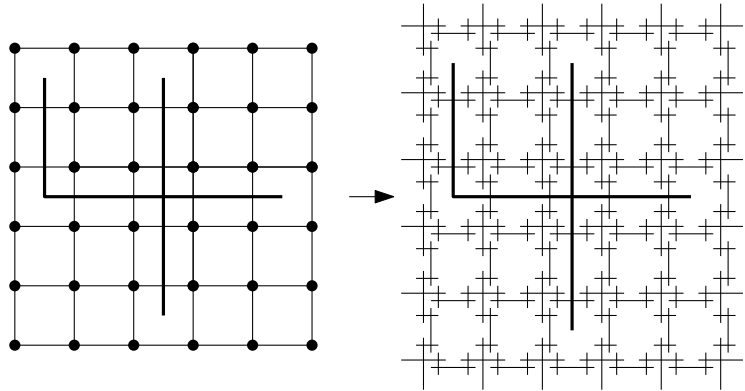}
    \caption{An example of the construction in the proof of Noodle-Forcing Lemma}
    \label{fig:noodleexample}
\end{figure}

The main idea of the proof of the Noodle-Forcing Lemma is to overlay the representation $R$ with a 
sufficiently fine grid-like mesh of horizontal and vertical segments, as shown in 
Figure~\ref{fig:noodleexample}. The relevant details of the construction are presented \arxivorgd{in the appendix}{in the full version~\cite{ChmelJelinekArXivVersion}}. Henceforth, we will use the notation $G^\#(R)$ for the specific graph created by the construction illustrated in Figure~\ref{fig:noodleexample}. 

Importantly, we can show that if the Noodle-Forcing Lemma is applied to a $\precg{k}$-string 
representation $R$, then the graph $G^\#(R)$ is also in $\precg{k}$-string.

\begin{lemma}\label{lem:preciserepresentationandnoodles}
    For all $k\geq 1$, given a $\precg{k}$-string representation $R$ of a graph $G$, the graph 
$G^\#(R)$ has a $\precg{k}$-string representation as well.
\end{lemma}
\begin{figure}
    \centering
    \input{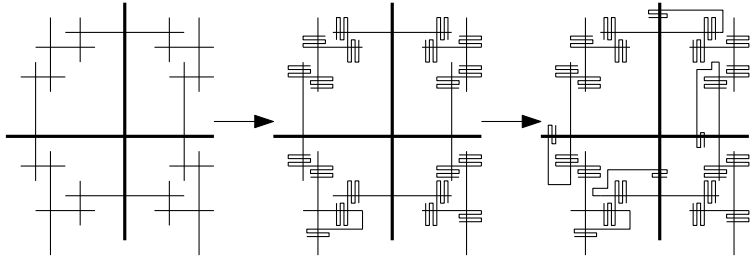}
    \caption{The extension of the representation $R^\#$ into a $\precg{k}$-string one in two steps}
    \label{fig:noodleextension}
\end{figure}
We postpone the proof \arxivorgd{to the appendix}{to the full version~\cite{ChmelJelinekArXivVersion}}; the idea is shown in Figure~\ref{fig:noodleextension}.

\subsubsection{Proofs of the non-inclusions}\label{sec:noninclusions}

Proofs involving the Noodle-Forcing Lemma usually follow the same pattern. The goal is typically to 
find a graph that belongs to one class, say \textsc{\precg{\ell}-string}, but not to another class, 
say \textsc{\precg{k}-string}. We thus begin by choosing a suitable gadget $G$ with an  
\precg{\ell}-string representation $R$ and apply the lemma to obtain the graph $G^\#=G^\#(R)$, which 
by Lemma~\ref{lem:preciserepresentationandnoodles} also belongs to \textsc{\precg{\ell}-string}. 
The main part is then to show that $G^\#$ is not in \textsc{\precg{k}-string}. This part of the 
argument will depend on the particular structure of $G$ and $R$, but there are certain common
ideas which we now present.

For contradiction, we assume that a $\precg{k}$-string representation $R'$ of $G^\#$ exists. By the 
Noodle-Forcing Lemma, we transform $R'$ into a representation $R^\varepsilon$ of $G^\#$ with 
properties as in the statement of the lemma. Moreover, $R^{\varepsilon}$ is also a $\precg{k}$-string 
representation as homeomorphisms and circular inversions preserve intersections between curves.

For brevity, we say that we \emph{precook} a $\precg{k}$-string representation $R^{\varepsilon}$ from $R'$ whenever we apply the preceding procedure with $\varepsilon>0$ small enough so that the properties in the following paragraphs hold.

The curve $R^{\varepsilon}(v)$ is $\varepsilon$-close to $R(v)$, i.e.,  
$R^{\varepsilon}(v)$ is confined into the set $N_\varepsilon(v):=\{x\colon \exists y\in 
R(v)\colon\mathrm{dist}(x,y)<\varepsilon\}$.
Following Chaplick et~al.~\cite{ChaplickEtAlNoodles}, we call the 
set $N_\varepsilon(v)$ the \emph{noodle} of~$v$. For small enough $\varepsilon$, each noodle is a 
simply connected region, and if two curves $R(u)$ and $R(v)$ intersect in $\ell$ points $p_1, 
p_2, \dotsc, p_\ell$, then $N_\varepsilon(u)$ and $N_\varepsilon(v)$ intersect in $\ell$ pairwise 
disjoint parallelograms $Z_1,\dotsc, Z_\ell$, where $Z_i$ contains $p_i$ (recall that in a proper 
representation, no crossing point may coincide with a bend of its curve). We call the sets 
$Z_1,\dotsc,Z_\ell$ the \emph{zones} of $N_\varepsilon(u)\cap N_\varepsilon(v)$.

We also assume that $\varepsilon$ is small enough so that the distance between any intersection point 
in $R$ and any endpoint of a curve in $R$ is strictly larger than $\varepsilon>0$. This is possible 
since $R$ is proper.
    
Consider now the intersection of the curve $R^\varepsilon(v)$ with a zone $Z_i$: its connected 
components shall be called the \emph{fragments} of $v$ in $Z_i$; see Figure~\ref{fig:zone}. Each 
fragment is a curve which either connects two opposite sides of $Z_i$ (we call such fragment a 
\emph{$v$-traversal} through $Z_i$) or has both endpoints on the same side of $Z_i$ (such a fragment 
is called a \emph{$v$-reversal}). Choosing $\varepsilon$ small enough, we may ensure that each $Z_i$ 
has at least one $v$-traversal, since $R^\varepsilon(v)$ must reach a point $\varepsilon$-close to 
each endpoint of $R(v)$, and the endpoints of $R(v)$ can be made at least $\varepsilon$-far from any 
zone.

\begin{figure}
    \centering
    \includegraphics[scale=0.95]{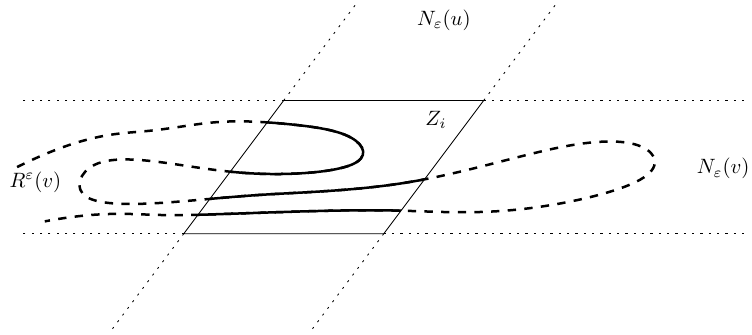}
    \caption{An example of a zone $Z_i$ at the intersection of two noodles $N_\varepsilon(u)$ and  
$N_\varepsilon(v)$, in which the curve $R^\varepsilon(v)$ forms one $v$-reversal and two 
$v$-traversals.}
    \label{fig:zone}
\end{figure}

Note that any $v$-traversal partitions its zone into two parts, separating the endpoints of  
any $u$-traversal of that zone from each other. Thus, any $v$-traversal must intersect any 
$u$-traversal of the same zone, and in fact they must intersect in an odd number of points. In 
particular, each zone has at least one intersection between two traversals. Moreover, any 
$v$-reversal in $Z_i$ forms a closed loop together with the segment on the boundary of $Z_i$ 
connecting its endpoints, and any $u$-fragment has both endpoints outside this loop. It follows that 
any $v$-reversal is intersected an even number of times (possibly zero) by any $u$-fragment.

As an example of this pattern, we show that there exists a 
$\precg{k+1}$-string graph that is not a $\precg{k}$-string graph, in fact it is not even a 
$k$-string graph. We note that this implies that the inclusions in Propositions 
\ref{prop:preciselystringplustwo} and \ref{prop:preciselystringtimesfour} are strict.

\begin{theorem}\label{thm:preciselykplusonenotink}
    For all $k\geq 1$, \textsc{$\precg{k+1}$-string} is not a subclass of \textsc{$k$-string},
    and therefore it is not a subclass of \textsc{$\precg{\ell}$-string} for any $1\leq \ell\leq k$ either.
\end{theorem}
\begin{proof}
    We construct a representation $R_{k+1}$ of a graph $G=K_2$, the complete graph on two vertices.
    We note that the construction is the same as in the proof of 
$\mathrm{B}_k\mathrm{-VPG}\subsetneq\mathrm{B}_{k+1}\mathrm{-VPG}$ by Chaplick 
et~al.~\cite{ChaplickEtAlNoodles}.
    As shown in Figure~\ref{fig:sausage}, the representation is $\precg{k+1}$-string, and we apply 
the Noodle-Forcing Lemma on the representation to obtain a graph $G^\#$ with representation 
$R_{k+1}^\#$.
    By Lemma~\ref{lem:preciserepresentationandnoodles}, the graph $G^\#$ has a $\precg{k+1}$-string 
representation.

    \begin{figure}
        \centering
        \input{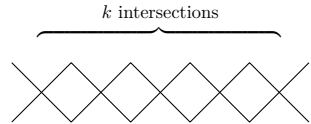}
        \caption{The sausage construction of representation $R_{k}$}
        \label{fig:sausage}
    \end{figure}

    We now claim that the graph $G^\#$ has no $k$-string representation.
    For contradiction, let $R'$ be a $k$-string representation of $G^\#$.
    We precook $R'$ into a $k$-string representation $R^\varepsilon$.

    Let $u,v$ be the two vertices of~$G$. Their two noodles $N(u)$ and $N(v)$ have $k+1$ zones in 
their intersection, and as observed before, each zone $Z_i$ must have at least one $u$-traversal and 
one $v$-traversal, which must intersect in at least one point.
    Therefore, there are at least $k+1$ intersection points between $R^\varepsilon(u)$ and 
$R^\varepsilon(v)$, contradicting the assumption that $R^\varepsilon$ is a 
$k$-string representation.

    The second part of the statement follows immediately, as any $\precg{\ell}$-string representation
of the graph $G^\#$ for $1\leq \ell\leq k$ is also a $k$-string representation.
\end{proof}

Our next goal is to show non-inclusion in the opposite direction and construct a $\precg{k}$-string 
graph that does not have a $\precg{k+1}$-string representation. The key role in our argument is 
played by the concept of \emph{faithful extension} of a string representation~$R$. Informally, a 
faithful extension of $R$ is obtained by extending each string $R(v)$ of $R$ by 
attaching two new (possibly empty) parts to it, each starting from an endpoint of the original 
string $R(v)$ and following it very closely, with the effect of ``doubling'' some number 
of intersections which appear along~$R(v)$ near its endpoints; see 
Figure~\ref{fig:suitableextensionexample}. The fully formal definition of faithful extension is 
rather technical, and is presented in the \arxivorgd{appendix}{full 
version~\cite{ChmelJelinekArXivVersion}}.

To construct a $\precg{k}$-string 
graph that does not have a $\precg{k+1}$-string representation, we proceed in two steps.
First, we show that given a $\precg{k}$-string representation $R$ of a graph $G$, the
$\precg{k}$-string graph $G^\#=G^\#(R)$ has a $\precg{k+1}$-string representation if and only 
if the representation $R$ can be faithfully extended into a $\precg{k+1}$-string representation; see 
Lemma~\ref{lem:extendingpreciserepresentations}.
In the second step, we find a graph with a $\precg{k}$-string representation that 
cannot be faithfully extended into a $\precg{k+1}$-string representation; see 
Lemma~\ref{lem:nonextendablepreciserepresentations}.

\begin{figure}
    \centering
    \input{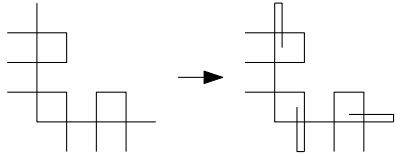}
    \caption{An example of a faithful extension of a $\precg{2}$-string representation into a 
$\precg{3}$-string representation}
    \label{fig:suitableextensionexample}
\end{figure}

\begin{lemma}\label{lem:extendingpreciserepresentations}
    For all $k\geq 1$, given a $\precg{k}$-string representation $R$ of a graph $G$, the 
graph $G^\#=G^\#(R)$ has these two properties:
\begin{enumerate}
\item $G^\#$ has a $\precg{k}$-string representation, and 
\item $G^\#$ has a $\precg{k+1}$-string representation if and only if the representation 
$R$ can be faithfully extended into a $\precg{k+1}$-string representation $R^+$ of~$G$.
\end{enumerate}
\end{lemma}
The proof of the lemma is postponed \arxivorgd{to the appendix}{to the full 
version~\cite{ChmelJelinekArXivVersion}}.

\begin{lemma}\label{lem:nonextendablepreciserepresentations}
    For all $k\geq 1$, there exists a graph $G_k$ with a $\precg{k}$-string representation $R_k$ that cannot be faithfully extended into a $\precg{k+1}$-string representation.
\end{lemma}
Again, we postpone the proof \arxivorgd{to the appendix}{to the full version~\cite{ChmelJelinekArXivVersion}}.
The proof is based on simple case analysis of the representations of graphs in Figure~\ref{fig:noninclusiongraphs}, where the cases are split by possible faithful extensions of some of the strings involved.

\begin{figure}
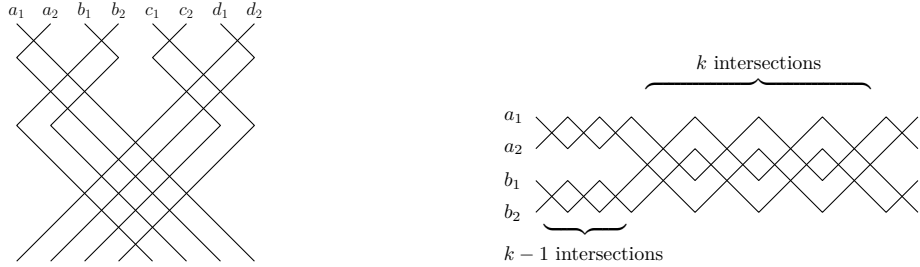

    \centering
    \input{img/noninclusiongraph1.tex}
    \input{img/noninclusiongraph2.tex}
    \caption{The graphs $G_1$ with representation $R_1$ (left), and $G_k$ with representation $R_k$ for $k\geq 2$ (right)}
    \label{fig:noninclusiongraphs}
\end{figure}

\begin{theorem}\label{thm:preciselyknotinkplusone}
    For all $k\geq 1$, \textsc{$\precg{k}$-string} $\not\subseteq$ \textsc{$\precg{k+1}$-string}.
\end{theorem}
\begin{proof}
    This follows immediately: by Lemma~\ref{lem:nonextendablepreciserepresentations}, there is a 
graph $G_k$ with its representation $R_k$ that cannot be faithfully extended, and therefore by Lemma~\ref{lem:extendingpreciserepresentations}, the graph $G^\#(R_k)$ has a $\precg{k}$-string 
representation, but it does not have a $\precg{k+1}$-string representation, proving the theorem.
\end{proof}

We have seen that a simple doubling argument shows that any $\precg{k}$-string graph is also a 
$\precg{4k}$-string graph. We will now show that this is best possible.

\begin{figure}
    \centering
    \includegraphics{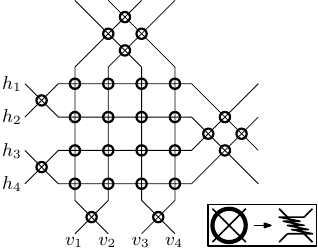}
    \caption{The representation $R$ used in the proof of Theorem~\ref{thm-noteven4k}. Each circled crossing represents $k$ crossings close to each other, where $k$ is odd.}
    \label{fig-noteven4k}
\end{figure}

\begin{theorem}\label{thm-noteven4k}
For every odd $k$, there is a $\precg{k}$-string graph $G_k$ which is not a $\precg{4k-2}$-string graph.
\end{theorem}
\begin{proof}
We use as the starting point the $\precg{k}$-string representation $R$ of the graph $K_8$ 
depicted in Figure~\ref{fig-noteven4k}. We will call the strings $v_1,\dotsc,v_4$ \emph{vertical}, 
and the remaining four strings \emph{horizontal}. Note that each circled intersection in that figure 
represents a $k$-tuple of intersections close to each other. 

Applying the Noodle-Forcing Lemma to the representation $R$, we obtain a $\precg{k}$-string graph 
$G_k=G^\#(R)$. We claim that $G_k$ has no $\precg{4k-2}$-string representation. For contradiction, assume 
that such a representation exists, and precook it into a representation $R^\varepsilon$, in which
every string $R^\varepsilon(u)$ is confined to a sufficiently narrow noodle $N(u)$ 
surrounding~$R(u)$. 

We shall 
use the following terminology: for any two strings $u$, $v$ of $R$, their \emph{crossing area}, 
denoted by $C(u,v)$, is the union of the $k$ zones formed by the intersections of their noodles, 
together with the parts of the noodles that lie between the zones. Thus, the crossing area forms a 
contiguous part of each of the two noodles. Removing the crossing area from a noodle $N(u)$ splits 
$N(u)$ into two connected pieces, which we call the \emph{tips} of $N(u)$ defined by $C(u,v)$.   
For a vertex $u$, we say that its intersection with another vertex $v$ is 
\begin{itemize}
\item \emph{ambiguous}, if at least one endpoint of $R^\varepsilon(u)$ is in $C(u,v)$,
\item \emph{peripheral}, if both endpoints of $R^\varepsilon(u)$ are in the same tip of $N(u)$ 
defined by $C(u,v)$, and 
\item \emph{central}, if each endpoint of $R^\varepsilon(u)$ is in a different tip of $N(u)$ 
defined by $C(u,v)$.
\end{itemize}
Notice that if the intersection with $v$ is peripheral for $u$, then every zone in $C(u,v)$ has an 
even number of $u$-traversals. Similarly, if the intersection is central for $u$, then 
every zone in $C(u,v)$ has an odd number of $u$-traversals.
We make the following observations:
\begin{enumerate}
\item For two vertices $u,v$ their mutual intersection cannot be peripheral for both of them, since 
this would mean that each zone in $C(u,v)$ has at least four crossing points of $R^\varepsilon(u)$ 
and $R^\varepsilon(v)$, which is impossible in a $\precg{4k-2}$-string representation.
\item It is also impossible for the intersection of $u$ and $v$ to be central for both vertices, 
since then every zone would contain an odd number of mutual crossing points.
\item Consider four distinct vertices $u, v, v', v''$, where the three crossing areas $C(u,v)$, 
$C(u,v')$ and $C(u,v'')$ appear in this order along the noodle $N(u)$. If neither $C(u,v)$ nor 
$C(u,v'')$ is peripheral for the vertex $u$, then $C(u,v')$ is central for~$u$.
\end{enumerate}
Referring to Figure~\ref{fig-noteven4k}, consider the crossing of $v_1$ and $v_2$. Since it cannot 
be peripheral for both curves by the first observation above, suppose w.l.o.g. that it is 
non-peripheral for~$v_1$. Likewise, suppose that the crossing of $v_3$ and $v_4$ is non-peripheral 
for $v_3$. We then look at the intersection of $v_1$ and $v_3$, and let $v_i\in\{v_1,v_3\}$ be the 
vertex for which the intersection is non-peripheral. It follows from the third observation that 
$v_i$ has central intersections with all the four horizontal vertices $h_1,\dotsc,h_4$. By analogous 
reasoning, a horizontal vertex $h_j\in\{h_1,\dotsc,h_4\}$ has central intersections with all the 
four vertical vertices, including~$v_i$. This however means that $R^\varepsilon(v_i)$ and 
$R^\varepsilon(h_j)$ have an odd number of mutual crossings by the second observation, a 
contradiction.
\end{proof}

We now turn our attention to the class \textsc{odd-string}.

\begin{theorem}\label{thm:preciselynonoddstring}
    There is a $\precg{2}$-string graph which is not an odd-string graph.
\end{theorem}
\begin{proof}
    Let us consider the graph $G$ with representation $R$ as in Figure~\ref{fig:notoddstringrepresentation}.
    We apply the Noodle-Forcing Lemma on $R$, getting the graph $G^\#$. We will show that $G^\#$ is 
not an odd-string graph.

For contradiction, let $R'$ be an odd-string representation of $G^\#$, which we 
precook into a representation $R^{\varepsilon}$.  Let $N(u)$ be the noodle of the vertex~$u$. 
Note that any pair of intersecting noodles $N(u)$, $N(v)$ forms two zones. The \emph{crossing area} 
of $N(u)$ and $N(v)$, denoted by $C(u,v)$, is then the union of the two zones together with the 
sections of $N(u)$ and of $N(v)$ that lie between the two zones. 

For a pair of intersecting noodles $N(u)$ and $N(v)$, we say that the curve $R^{\varepsilon}(u)$ 
\emph{covers} the intersection of $N(u)$ and $N(v)$, if $R^{\varepsilon}(u)$ has at least one 
endpoint in the crossing area $C(u,v)$.

We claim that for any pair of intersecting noodles $N(u)$ and $N(v)$, at least one of the two curves 
$R^{\varepsilon}(u), R^{\varepsilon}(v)$ must cover their intersection. Indeed, suppose for 
contradiction that neither of the two curves covers the intersection, and let $Z$ and $Z'$ be the 
two zones of $N(u)\cap N(v)$. Then either both $Z$ and $Z'$ have an odd number of $u$-traversals or 
both $Z$ and $Z'$ have an even number of $u$-traversals, depending on the positions of the 
two endpoints of $R^\varepsilon(u)$. The same is true for $v$-traversals as well. Thus, the number 
of crossings between $R^{\varepsilon}(u)$ and $R^{\varepsilon}(v)$ inside $Z$ has the same parity as 
the number of such crossings inside $Z'$, showing that the two curves have an even number of 
crossings overall. This is impossible, since $R^\varepsilon$ is an odd-string representation. This 
shows that at least one of the curves covers the intersection.

Now, we note that each of the two bold curves can only cover its crossings with at most two of the thin
 curves. Therefore, there are two thin curves $R^{\varepsilon}(x), R^{\varepsilon}(y)$ that do 
not have their crossings with the bold curves covered by the bold curves, and therefore the thin
curves have to cover the bold curve crossings with their endpoints.
    But this implies that neither of the two curves $R^{\varepsilon}(x), R^{\varepsilon}(y)$ can 
cover the crossing of $N(x)$ and $N(y)$, which is a contradiction.
\end{proof}

\begin{figure}
    \centering
    \input{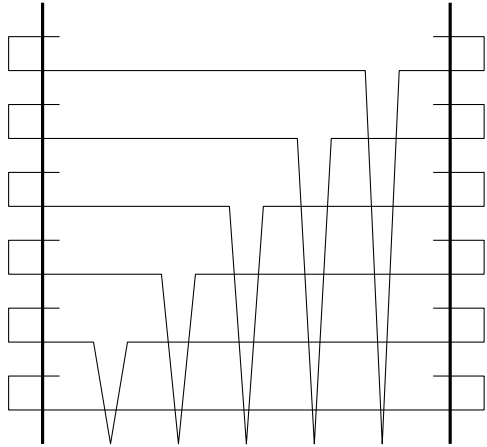}
    \caption{The graph $G$ with a $\precg{2}$-string representation for Theorem~\ref{thm:preciselynonoddstring}}
    \label{fig:notoddstringrepresentation}
\end{figure}

This result shows that the class \textsc{odd-string} differs from the class \textsc{string}.
Moreover, we can summarize our findings by the following proposition.
\begin{proposition}
    For $k, \ell\in\bbN$, $\text{\textsc{$\precg{k}$-string}}\subseteq\text{\textsc{$\precg{\ell}$-string}}$ if and only if one of the two conditions holds
    \begin{itemize}
        \item $k\leq \ell$ and $k,\ell$ have the same parity, or
        \item $k$ is odd, $\ell$ is even and $\ell\geq 4k$.
    \end{itemize}
\end{proposition}
\begin{proof}
    The backward implication follows from Propositions~\ref{prop:preciselystringplustwo} and~\ref{prop:preciselystringtimesfour}.
    The forward implication follows from Theorems~\ref{thm:preciselyknotinkplusone} (which removes any case with $\ell<k$),~\ref{thm:preciselynonoddstring} (which removes any case where $k$ is even and $\ell$ is odd), and~\ref{thm-noteven4k} (which removes any case where $k$ is odd, $\ell$ is even and $\ell\leq 4k-2$).
\end{proof}

\section{Final remarks and open problems}

\subsection{$\precg{\ell}$-string graphs from $k$-string graphs}
It would be nice to clarify the inclusions between the classes \textsc{$k$-string} and 
\textsc{\precg{\ell}-string} for various values of $\ell$ and $k$. The key problem is to determine,
for a given $k$, the smallest $\ell$ such that \textsc{$k$-string} is a subclass of 
\textsc{\precg{\ell}-string}. 

Immediately from Theorem~\ref{thm:preciselyknotinkplusone}, we can see that (for $k>1$) $\ell$ must 
be at least $k+2$, since $k$-string graphs contain both the class of $\precg{k-1}$-string graphs and 
$\precg{k}$-string graphs, and therefore are not a subclass of either $\precg{k}$-string or 
$\precg{k+1}$-string. On the other hand, our current methods get no better upper bound than $\ell\le 
8k$.

To establish the bound, we first define the class of proper $k$-string graphs to be the class of all 
$k$-string graphs that have a proper $k$-string representation. Our proof works in two steps: in the 
first step, we build a proper $2k$-string representation from a $k$-string representation, and in the 
second step, we build a $\precg{4\ell}$-string representation from a proper $\ell$-string 
representation. We obtain the following result, whose proof appears 
\arxivorgd{in the appendix}{in the full version~\cite{ChmelJelinekArXivVersion}}. We do not know 
if the bound $8k$ is tight.

\begin{theorem}\label{thm:stringtopreciselykstring}
    For all $k\geq 1$, \textsc{$k$-string} is a subclass \textsc{$\precg{8k}$-string}.
\end{theorem}

\subsection{Complexity}
After defining the new classes of odd-string graphs and $\precg{k}$-string graphs, we may naturally 
ask about the complexity of their recognition. We can easily see that all mentioned classes contain 
all grid intersection graphs, i.e., the intersection graphs of horizontal and vertical segments (for 
$\precg{2}$-string graphs, we use Proposition~\ref{prop:preciselystringtimestwoforbipartite}) and are 
contained in string graphs. Therefore, by the sandwiching results of Kratochvíl~\cite{KratochvilNPhard}, or Mustaţă and 
Pergel~\cite{MustataPergelRecognition}, we see that the recognition problem of \textsc{odd-string} as 
well as of \textsc{$\precg{k}$-string} is NP-hard.

For $\precg{k}$-string graphs with fixed $k$, we can easily see that the problem is in NP (and 
hence NP-complete). However, when $k$ is a part of the input, it is not clear whether the problem 
still remains in NP. 

For the recognition of \textsc{odd-string}, it is not even clear whether the problem is decidable, 
as there is no known upper bound on the number of intersections an odd-string graph on $n$ vertices 
may require in its representation. This situation is reminiscent of a similar issue with the 
recognition of string graphs, whose decidability used to be an open problem until Schaefer and 
Štefankovič first showed that the problem is decidable~\cite{SchaeferStefankovicStringDecidable}, and 
later with Sedgwick that it belongs to NP~\cite{SchaeferSedgwickStefankovicStringInNP}. 
However, we are unable to adapt their methods to odd-string representations. 

Another type of complexity questions focuses on  the complexity of deciding the existence of a 
$\precg{k}$-string representation when we are given a $\precg{\ell}$-string representation as the 
input. The following cases seem interesting:
\begin{itemize}
    \item $k=\ell+1$ (i.e., can we extend the representation to have one more intersection per pair 
of curves? It also makes sense to restrict this question to the faithful extensions \arxivorgd{from Definition~\ref{def-faithful} in the appendix}{fully defined in the full version~\cite{ChmelJelinekArXivVersion}}),
    \item $k=\ell-2$ (i.e., can we remove two intersections per pair, or is this the ``minimal'' representation of this graph with respect to the number of intersection points per pair with the given parity?),
    \item as a generalization of the previous case, $k=\ell-c$ for any $c\in\mathbb{N}$.
\end{itemize}
We remark that analogous questions have previously been tackled, e.g., for the $B_k$-VPG hierarchy, 
where Chaplick et al.~\cite{ChaplickEtAlNoodles} showed that recognizing $B_k$-VPG graphs is 
NP-complete, even when a $B_{k+1}$-VPG representation is given as part of the input.

\bibliographystyle{splncs04}
\bibliography{main}

\arxivorgd{}{\end{document}} 

\appendix
\section{Proofs of the hierarchy inclusions}
\begin{proof}[Proposition~\ref{prop:preciselystringplustwo}]
    We build a $\precg{k+2}$-string representation from a $\precg{k}$-string representation.

    For every pair of intersecting curves $c_1,c_2$, we choose a single intersection point $P$ of the two curves.
    As we assume the representations to be proper, there exists a sufficiently small 
neighborhood of $P$ that contains only parts of $c_1, c_2$ with $P$ as their only 
intersection point and no other curve.
    We change the representation in this neighborhood of $P$ so that the curves 
intersect three times -- $c_2$ will remain the same, while we change $c_1$ to cross $c_2$ three 
times as in Figure~\ref{fig:preciselystringchangeplustwo}.
\end{proof}

\begin{proof}[Proposition~\ref{prop:preciselystringtimesfour}]
    Given a $\precg{k}$-string representation $R$ of $G$, we construct a $\precg{4k}$-string 
representation $R'$ of $G$.
    Intuitively, we ``double'' each string of $R$, which quadruples the number of intersections.

    As the representation is proper, for each curve $c$ in the representation, there exists an $\varepsilon_c>0$ such that in the $\varepsilon_c$-neighborhood of $c$, no two other curves intersect, and if another curve in the representation enters the neighborhood, then it intersects $c$ precisely once and then leaves the neighborhood.
    We then take $\varepsilon:=\min_{c\in R}\varepsilon_c$.
    
     We extend each curve $c$ as follows: we take the curve $c$, draw another curve $c'$ parallel to 
$c$ inside the $\varepsilon/2$-neighborhood and join one of the two pairs of $\varepsilon/2$-close 
endpoints together by a single line segment, as shown in Figure~\ref{fig:preciselystringchangetimesfour}.
    It is clear that given two curves $c,d$ of $R$ intersecting in the point $P$, the two new 
doubled curves have all four intersections in the $\varepsilon$-neighborhood of $P$: the 
intersections between $c$ and $d$, $c'$ and $d$ and $d'$ and $c$ exist immediately by the 
construction.
    The intersection between $c'$ and $d'$ exists as the curves intersect the boundary of the $\varepsilon$-neighborhood in the cyclic order $c, c', d, d', c', c, d', d$, and therefore $c'$ and $d'$ must intersect as well.

    No more intersections are created anywhere else, which ensures that we indeed have a $\precg{4k}$-string representation.
\end{proof}
\begin{proof}[Proposition~\ref{prop:preciselystringtimestwoforbipartite}]
    We repeat the construction from the proof of the previous proposition, however, we only apply the ``doubling'' operation on the curves representing the vertices of a single partition of $G$.
    This ensures that there are precisely $2k$ intersection points between two curves with a non-empty intersection.
\end{proof}
\section{Proof of Lemma 2}
We briefly sketch the construction of $G^\#$ and $R^\#$, which works in two steps: in the first step, we overlay the proper string representation $R$ of a graph $G$ by a plane grid graph $H$ with particular properties, and in the second step, we construct a string representation based on $R$ and $H$ with the additional property that all newly added strings are horizontal or vertical line segments.

The first step is performed as follows.
We start by defining \emph{special} points of $R$ -- these are the endpoints of the curves, bend points of the curves, and intersection points of the curves.
We then construct the plane grid graph $H$ that overlays the representation with the following properties:
\begin{itemize}
    \item[(P1)] The edges of $H$ are drawn as vertical and horizontal segments, and every internal face of $H$ is a rectangle. Moreover, the outer face of $H$ does not intersect any curve of $R$.
    \item[(P2)] No curve of $R$ passes through a vertex of $H$ and no edge of $H$ passes through a special point of $R$.
    \item[(P3)] Every face of $H$ contains at most one special point of $R$ and no two faces containing a special point are adjacent.
    \item[(P4)] Every edge of $H$ has at most a single intersection with the curves of $R$.
    \item[(P5)] Every face of $H$ intersects at most two curves of $R$, and if a face $f$ intersects exactly two curves of $R$, then the two curves intersect inside the face $f$.
    \item[(P6)] Every curve of $R$ intersects the boundary of a face of $H$ at most twice.
\end{itemize}
The first three properties can be obtained by taking a grid that is large enough and fine enough.
The remaining three properties are then ensured by suitably splitting the faces of $H$; see Chaplick 
et al.~\cite{ChaplickEtAlNoodles} for details.

We proceed with the second step.
Given the plane graph $H$, we build a representation $R^\#$ of the graph $G^\#$ as follows.
Starting with the vertices of $H$, each $v\in V(H)$ is converted into two vertices $S_1(v), S_2(v)$ 
connected by an edge.
Each edge $\{u,v\}=e\in E(H)$ that is incident with $v$ is converted into three vertices $S(v,e), S(u,e), S(e)$ with edges $\{S(v,e), S(e)\},\{S(u,e),S(e)\}$.
If an edge $e$ that is incident to a vertex $v$ is drawn as a horizontal segment, we add an edge 
between $S(v,e)$ and $S_1(v)$ and if $e$ is drawn as a vertical segment, the edge is added between 
the vertices $S(v,e)$ and $S_2(v)$.
We can represent this naturally with a grid intersection representation.
To finish the construction, we add the vertices of $G$ as in the representation $R$ with edges given by the intersections.
An example of the construction is shown in Figure~\ref{fig:noodleexample}.

\begin{proof}
    Let us take the representation $R^\#$ as in the proof of the Noodle-Forcing Lemma and add 
the additional required intersections. We will refer to the curves in $R$ as the \emph{original 
curves}, and the curves in $R^\#\setminus R$ as \emph{new curves}. Note that the new curves are 
actually segments.

Since the original curves already form a $\precg{k}$-string representation, 
    we only have to take care of the intersections between two new curves, and the 
intersections between an original curve and a new one.
    We first take care of the intersections between pairs of new curves.
    The new curves induce three types of edges: the edges $\{S_1(v), S_2(v)\}$, $\{S(v,e), 
S_i(v)\}$, and $\{S(e),S(v,e)\}$.
    In the representation $R^\#$, any such edge corresponds to a single intersection point, therefore, 
we have to add $k-1$ more.
    In the case of edges $\{S(v,e), S_i(v)\}$ and $\{S(e),S(v,e)\}$, we extend the segment $S(v,e)$, 
so that on the extension on one endpoint, we add the intersections with $S(e)$ and on the extension 
of the other endpoint, we add the required intersections with $S_i(v)$.
    The edges $\{S_1(v), S_2(v)\}$ are resolved similarly, we extend one endpoint of $S_1(v)$ to add 
the required intersections. This is depicted as the first step in Figure~\ref{fig:noodleextension}.

    Finally, we take care of the intersections between the new curves and the original ones.
    The only edges formed by such intersections are of the form $\{S(e), v\}$ for $v\in V(G)$.
    To add the required additional intersections, we extend the strings representing the vertices 
$S(e)$ from one of the endpoints.
    The is the second step in Figure~\ref{fig:noodleextension}. This yields a $\precg{k}$-string 
representation of~$G^\#$.
\end{proof}
\section{Proof of Lemma 3}
Before stating the proof of the lemma, we first formally define the faithful extension.
\begin{definition}[Faithful extension]\label{def-faithful}
    Let $R$ be a proper string representation of a graph $G=(V,E)$ and let $R(v)$ be one of its strings. 
    Let $p_1,p_2,\dotsc,p_\ell$ be the intersection points on $R(v)$, in the order in which they appear 
    on it. Let $a,b\ge 0$ be integers such that $a+b\le \ell$. A \emph{faithful extension} of the curve 
    $R(v)$ with respect to $R$ is a curve $R^+(v)$ which contains $R(v)$ as a subset, and which 
    contains, in the order from beginning to end, $a+\ell+b$ intersection points $p'_a, 
    p'_{a-1},\dotsc,p'_1,p_1,p_2,\dotsc, p_\ell, p'_\ell,p'_{\ell-1},\dotsc,p'_{\ell-b+1}$ with the 
    following property: for any $i\in\{1,2,\dotsc,a\}\cup \{\ell-b+1,\ell-b+2,\dotsc,\ell\}$, if $R(u)$ 
    is the string intersected by $R(v)$ in $p_i$, then $R(u)$ is intersected by $R^+(v)$ in $p'_i$ and 
    moreover, $R(u)$ contains no other intersection between the two points $p_i$ and $p'_i$.
    We say that $R^+(v)$ is \emph{duplicating} the intersections $p_1,p_2,\dotsc,p_a$ and $p_\ell, 
    p_{\ell-1},\dotsc,p_{\ell-b+1}$ of~$R(v)$. 
    
    A representation $R^+$ of the graph $G$ is a \emph{faithful extension} of the 
    representation $R$, if
    \begin{enumerate}
    \item each string $R^+(v)$ is a faithful extension of $R(v)$ with respect to $R$,
    \item each intersection of two distinct strings $R(u)$ and $R(v)$ of $R$ is duplicated by at 
    most one of the two strings $R^+(u)$ and $R^+(v)$,
    \item for distinct vertices $u,v$, any intersection between $R^+(u)$ and $R^+(v)$ is either an 
    intersection of $R(u)$ and $R(v)$ or a point duplicating such an intersection; in particular, 
    $(R^+(u)\setminus R(u))$ and $(R^+(v)\setminus R(v))$ are disjoint, and
    \item for any $\varepsilon >0$, there is a homeomorphism of the plane which transforms $R^+$ in such 
    a way that each $R^+(v)$ is mapped into an $\varepsilon$-neighborhood of $R(v)$, while  
    $R(v)$ is mapped onto itself.
    \end{enumerate}
\end{definition}
\begin{proof}
    Given a $\precg{k}$-string representation $R$ of a graph $G$, we apply the Noodle-Forcing Lemma 
to the representation $R$ and we obtain a graph $G^\#$ with its representation $R^\#$ as described in 
the Noodle-Forcing Lemma. By Lemma~\ref{lem:preciserepresentationandnoodles}, $G^\#$ has a 
$\precg{k}$-string representation. It remains to show that it also satisfies the second property 
of Lemma~\ref{lem:extendingpreciserepresentations}.

    We first prove the reverse implication.
    Let $R^+$ be a $\precg{k+1}$-string representation of $G$ that faithfully extends $R$.
    We can then take the representation $R^\#$ of $G^\#$ and replace each string representing a vertex 
of $G$ by its extended version from $R^+$, drawn in a sufficiently small neighborhood of the 
original string, so that the extended string intersects the same strings as the original one.
    This forms a representation $\tilde{R}$ of $G^\#$, in which all pairs of distinct vertices of $G$ 
are represented by curves that intersect in either $0$ or $k+1$ points. We may also draw the faithfully extended strings in such a way that a string 
$\tilde R(v)$ representing a vertex from $G$ intersects a string representing a vertex from
$V(G^\#)\setminus V(G)$ in at most two points. We then use the same approach as in the proof of 
Lemma~\ref{lem:preciserepresentationandnoodles} to add the required number of intersections to the 
strings representing the vertices from $V(G^\#)\setminus V(G)$.

    We now prove the forward implication. 
    Suppose we are given a $\precg{k+1}$-string representation $R'$ of $G^\#$, and we want to 
show that it yields a representation $R^+$ that faithfully extends $R$.
    We then precook $R'$ into a $\precg{k+1}$-string representation $R^{\varepsilon}$.

Since there are $k$ zones determined by the intersections of $u$ and $v$, while $R^\varepsilon(u)$ 
and $R^\varepsilon(v)$ have $k+1$ intersections, it follows that there is a unique zone $Z_i$ which 
contains exactly two intersections between the two curves. There are only two possibilities how this 
might occur: either $Z_i$ has a single $u$-traversal and two $v$-traversals (we then say that $p_i$ 
is \emph{duplicated by~$v$}), or $Z_i$ has two $u$-traversals and one $v$-traversal (and we say that 
$p_i$ is \emph{duplicated by~$u$}).

Consider now a vertex $v$ of $G$, and let $p_1,p_2,\dotsc,p_\ell$ be all the intersections on $R(v)$ 
in the representation $R$, from beginning to end, and let $Z_1,\dotsc,Z_\ell$ be their 
corresponding zones. We now show that if $v$ duplicated an intersection $p_j$, then it must have 
also duplicated all the intersections before it or all the intersections after it. For 
contradiction, suppose that there are indices $i<j<k$ such that $p_j$ is duplicated by $v$, but 
$p_i$ and $p_k$ are not, as in Figure~\ref{fig:zonescontradiction}, 
and consider the disjoint regions $A,B,C,D$ as depicted there.
    Each of the four regions has an odd number of crossings with $R^\varepsilon(v)$ on the boundary, 
and therefore, there must be an endpoint of $R^\varepsilon(v)$ in each of the four regions.
    However, curves only have two endpoints, and this yields a contradiction.
    Therefore, the zones with intersections duplicated by $v$ must form two (possibly empty) 
consecutive intervals $Z_1,\ldots,Z_m$ for some $m\ge 0$ and $Z_{k},\ldots,Z_\ell$ for some $k>m$. 
By using the information about which intersection is duplicated by which vertex, we may now 
straightforwardly construct a faithful extension $R^+$ of $R$ which is a $\precg{k+1}$-string 
representation of~$G$, by confining each $R^+(v)$ inside the noodle defined by $R(v)$, and ensuring 
that $R^+(v)$ traverses each zone the same number of times as~$R^\varepsilon(v)$. 

    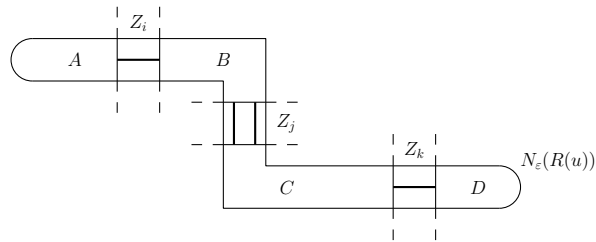
\begin{figure}
        \centering
        \tikzstyle{ipe stylesheet} = [
  ipe import,
  even odd rule,
  line join=round,
  line cap=butt,
  ipe pen normal/.style={line width=0.4},
  ipe pen heavier/.style={line width=0.8},
  ipe pen fat/.style={line width=1.2},
  ipe pen ultrafat/.style={line width=2},
  ipe pen normal,
  ipe mark normal/.style={ipe mark scale=3},
  ipe mark large/.style={ipe mark scale=5},
  ipe mark small/.style={ipe mark scale=2},
  ipe mark tiny/.style={ipe mark scale=1.1},
  ipe mark normal,
  /pgf/arrow keys/.cd,
  ipe arrow normal/.style={scale=7},
  ipe arrow large/.style={scale=10},
  ipe arrow small/.style={scale=5},
  ipe arrow tiny/.style={scale=3},
  ipe arrow normal,
  /tikz/.cd,
  ipe arrows, 
  <->/.tip = ipe normal,
  ipe dash normal/.style={dash pattern=},
  ipe dash dotted/.style={dash pattern=on 1bp off 3bp},
  ipe dash dashed/.style={dash pattern=on 4bp off 4bp},
  ipe dash dash dotted/.style={dash pattern=on 4bp off 2bp on 1bp off 2bp},
  ipe dash dash dot dotted/.style={dash pattern=on 4bp off 2bp on 1bp off 2bp on 1bp off 2bp},
  ipe dash normal,
  ipe node/.append style={font=\normalsize},
  ipe stretch normal/.style={ipe node stretch=1},
  ipe stretch normal,
  ipe opacity 10/.style={opacity=0.1},
  ipe opacity 30/.style={opacity=0.3},
  ipe opacity 50/.style={opacity=0.5},
  ipe opacity 75/.style={opacity=0.75},
  ipe opacity opaque/.style={opacity=1},
  ipe opacity opaque,
]
\definecolor{red}{rgb}{1,0,0}
\definecolor{blue}{rgb}{0,0,1}
\definecolor{green}{rgb}{0,1,0}
\definecolor{yellow}{rgb}{1,1,0}
\definecolor{orange}{rgb}{1,0.647,0}
\definecolor{gold}{rgb}{1,0.843,0}
\definecolor{purple}{rgb}{0.627,0.125,0.941}
\definecolor{gray}{rgb}{0.745,0.745,0.745}
\definecolor{brown}{rgb}{0.647,0.165,0.165}
\definecolor{navy}{rgb}{0,0,0.502}
\definecolor{pink}{rgb}{1,0.753,0.796}
\definecolor{seagreen}{rgb}{0.18,0.545,0.341}
\definecolor{turquoise}{rgb}{0.251,0.878,0.816}
\definecolor{violet}{rgb}{0.933,0.51,0.933}
\definecolor{darkblue}{rgb}{0,0,0.545}
\definecolor{darkcyan}{rgb}{0,0.545,0.545}
\definecolor{darkgray}{rgb}{0.663,0.663,0.663}
\definecolor{darkgreen}{rgb}{0,0.392,0}
\definecolor{darkmagenta}{rgb}{0.545,0,0.545}
\definecolor{darkorange}{rgb}{1,0.549,0}
\definecolor{darkred}{rgb}{0.545,0,0}
\definecolor{lightblue}{rgb}{0.678,0.847,0.902}
\definecolor{lightcyan}{rgb}{0.878,1,1}
\definecolor{lightgray}{rgb}{0.827,0.827,0.827}
\definecolor{lightgreen}{rgb}{0.565,0.933,0.565}
\definecolor{lightyellow}{rgb}{1,1,0.878}
\definecolor{black}{rgb}{0,0,0}
\definecolor{white}{rgb}{1,1,1}
\begin{tikzpicture}[ipe stylesheet,scale=0.5, every node/.style={scale=0.5}]
  \draw
    (96, 752)
     -- (272, 752)
     -- (272, 656)
     -- (448, 656);
  \draw
    (448, 624)
     -- (240, 624)
     -- (240, 720)
     -- (96, 720);
  \draw
    (96, 752)
     arc[start angle=90, end angle=270, radius=16];
  \draw
    (448, 656)
     arc[start angle=-90, end angle=90, x radius=16, y radius=-16];
  \draw
    (160, 752)
     -- (160, 720);
  \draw
    (192, 752)
     -- (192, 720);
  \draw
    (368, 656)
     -- (368, 624);
  \draw
    (400, 656)
     -- (400, 624);
  \draw
    (240, 704)
     -- (272, 704);
  \draw
    (240, 672)
     -- (272, 672);
  \draw[ipe pen heavier]
    (160, 736)
     -- (192, 736);
  \draw[ipe dash dashed]
    (160, 752)
     -- (160, 776);
  \draw[ipe dash dashed]
    (160, 720)
     -- (160, 696);
  \draw[ipe dash dashed]
    (192, 720)
     -- (192, 696);
  \draw[ipe dash dashed]
    (192, 752)
     -- (192, 776);
  \draw[ipe dash dashed]
    (240, 704)
     -- (216, 704);
  \draw[ipe dash dashed]
    (216, 672)
     -- (240, 672);
  \draw[ipe dash dashed]
    (272, 672)
     -- (296, 672);
  \draw[ipe dash dashed]
    (272, 704)
     -- (296, 704);
  \draw[ipe dash dashed]
    (368, 656)
     -- (368, 680);
  \draw[ipe dash dashed]
    (368, 624)
     -- (368, 600);
  \draw[ipe dash dashed]
    (400, 624)
     -- (400, 600);
  \draw[ipe dash dashed]
    (400, 656)
     -- (400, 680);
  \node[ipe node, anchor=center, font=\Large]
     at (128, 736) {$A$};
  \node[ipe node, anchor=center, font=\Large]
     at (240, 736) {$B$};
  \node[ipe node, anchor=center, font=\Large]
     at (288, 640) {$C$};
  \node[ipe node, anchor=center, font=\Large]
     at (432, 640) {$D$};
  \node[ipe node, font=\Large]
     at (464, 656) {$N_{\varepsilon}(R(u))$};
  \draw[ipe pen heavier]
    (368, 640)
     -- (400, 640);
  \draw[ipe pen heavier]
    (248, 704)
     -- (248, 672);
  \draw[ipe pen heavier]
    (264, 704)
     -- (264, 672);
  \node[ipe node, anchor=base, font=\Large]
     at (176, 760) {$Z_i$};
  \node[ipe node, anchor=west, font=\Large]
     at (280, 688) {$Z_j$};
  \node[ipe node, anchor=base, font=\Large]
     at (384, 664) {$Z_k$};
\end{tikzpicture}
        \caption{The impossible situation with zones in proof of Lemma~\ref{lem:extendingpreciserepresentations}}
        \label{fig:zonescontradiction}
    \end{figure}
\end{proof}

\section{Proof of Lemma 4}
\begin{proof}
    We start with the case $k=1$, where the graph is $G_1=K_8$, the complete graph on eight vertices, with its representation $R_1$ depicted in Figure~\ref{fig:noninclusiongraphs}.
    We want to show that there is no faithful extension of the representation and for contradiction, 
we assume that such extension exists.
    We will consider the added intersections of vertices $c_1,c_2,d_1,d_2$.

    First, we note that only one of $c_1,c_2$ and one of $d_1,d_2$ can be extended from the top endpoint farther than their intersection as otherwise, we would get more than two intersection points for the two vertices (let $c_x, d_y$ be the vertices that are extended from the top endpoint farther than their intersection with $c_{3-x}, d_{3-y}$ respectively).
    This also implies that at least one of the four intersections between $c_i,d_j$ cannot be added by the top extension -- in particular, the intersection between $c_{3-x}, d_{3-y}$.
    We note that some configurations of these extensions cannot be extended from the bottom (e.g., if $c_2$ is extended as far down as possible and $d_1$ is extended to its intersection with $c_1$, then the intersection point of $c_1,d_2$ cannot be reached from the bottom), however, such cases cannot be the extensions either way.

    We can also use the argument to see that there is at least one intersection point between $a_i$ and $b_j$ that is not covered by the extensions starting from the top.
    Therefore, at least one of $a_i,b_i$ must be extended from the bottom to obtain the uncovered intersection and the same is true for at least one $c_j,d_j$.
    However, each pair with one vertex $a_i$ or $b_i$ and the other vertex $c_j$ or $d_j$ has an intersection in the bottom half of the representation, and hence it cannot happen that both of them can be extended up to their so far uncovered intersection.
    Therefore, the representation cannot be faithfully extended.

    Next, we continue with the case $k\geq 2$.
    In this case, the graph is $G_k=K_4$, the complete graph on four vertices, with the representation $R_k$ as in Figure~\ref{fig:noninclusiongraphs}.
    We start by noting that at most one of the strings $a_1$ and $a_2$, or $b_1$ and $b_2$ respectively, can be extended as both the leftmost and the rightmost intersection point are between $a_1,a_2$, or $b_1,b_2$ respectively.
    Therefore, there are an $a_i$ and a $b_j$ that cannot be extended and therefore, the number of 
intersection points between $a_i$ and $b_j$ remains at $k$, showing that we cannot faithfully extend 
the representation.
\end{proof}
\section{Proof of Theorem 5}
\begin{proposition}
    For all $k\geq 1$, \textsc{$k$-string} $\subseteq$ \textsc{proper $2k$-string}.
\end{proposition}

\begin{proof}
    Given a graph $G$ with a $k$-string representation $R$, we start by creating a proper $2k$-string representation of $G$.
    First, we ensure each curve is simple by redrawing each curve in an $\varepsilon$-neighborhood of any self-crossing so that it becomes two arcs instead.
    Next, we note that the issue with possibly infinitely many intersection points could not have happened as we are already bounding the number of intersection points before.
    As a third step, we fix the issue that two curves may only touch and not cross each other.
    We do this by changing each touch into two crossings, which in total may double the number of intersection points.
    
    Formally, let $T$ be the intersection point where two curves touch (and do not cross).
    There is a $\varepsilon>0$ such that in the $\varepsilon$-neighborhood of $T$, there are no other intersection points in the representation.
    However, we note that there can possibly be multiple curves that intersect in $T$.
    We choose to redraw a curve $R(v)$ such that in the $\varepsilon$-neighborhood of $T$ divided into two regions by the curve $R(v)$, it does not happen that there are curves $R(v')$ and $R(v'')$ such that $R(v')$ only appears in one of the regions, $R(v'')$ appears only in the other region and the two curves touch $R(v)$ in $T$.
    (Such $R(v)$ can always be found: we start with any $R(u)$ touching in $T$, and if there are $R(u'), R(u'')$ that are each in just one of the two regions, we choose one of the two arbitrarily. Then, if $R(u')=R(w)$ still does not satisfy the conditions, we choose $R(w'), R(w'')$ that is in the same region as $R(u')$ when splitting the $\varepsilon$-neighborhood by $R(u)$. The region will always shrink, and as there are at most $n$ curves intersecting in the point, in a finite amount of steps, we will stop with some $R(v)$ that satisfies the conditions.)

    We then redraw $R(v)$ as per Figure~\ref{fig:redrawingtouches}, which adds one more intersection point for every curve $R(v)$ originally touched with both of the new intersection points being crossings.
    We have therefore removed at least a single touch, and therefore after a finite amount of such changes, we get a representation with no two curves touching.
    \begin{figure}
        \centering
        \input{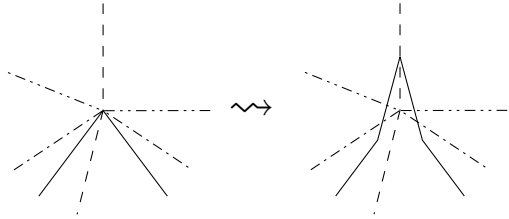}
        \caption{The process of redrawing touchings into crossings}
        \label{fig:redrawingtouches}
    \end{figure}

    Finally, we ensure that each crossing point is shared by at most two curves: this can be done by perturbing the curves slightly as now, curves may only cross and not touch in their intersection points.
\end{proof}

\begin{proposition}
    For all $\ell\geq 1$, \textsc{proper $\ell$-string} $\subseteq$ \textsc{$\precg{4\ell}$-string}.
\end{proposition}
\begin{proof}
    Starting with a proper $\ell$-string representation, we create a $4\ell$-string representation by ``doubling'' each string as in the proof of Proposition~\ref{prop:preciselystringtimesfour}.
    In such a representation, we know by the construction that every pair of strings that has a non-empty intersection intersects in $4k$ points for some $k\in\mathbb{N}$.
    In particular, the number of intersections is always even.

    Therefore, we may continue with the argument as in Proposition~\ref{prop:preciselystringplustwo} 
and for every pair of curves that has fewer than $4\ell$ intersection points, we add $2$ 
intersection points until it has precisely $4\ell$ intersection points.
    We note that we cannot skip $4\ell$ intersections as we start with an even number of intersections, and by adding two intersection points, the parity of the number of intersection points remains unchanged.
\end{proof}

Theorem~\ref{thm:stringtopreciselykstring} follows immediately from the two preceding propositions.
\end{document}